\documentclass[12pt]{amsart}
\usepackage[colorlinks=true,urlcolor=blue,linkcolor=red,citecolor=magenta]{hyperref}
\usepackage{amsfonts}
\usepackage{amsmath,amssymb,amsthm,latexsym,graphicx,graphics}
\usepackage{multicol}
\usepackage{eurosym}
\usepackage{amsfonts}
\usepackage{amsmath,amssymb,amsfonts}
\usepackage{amsfonts}
\usepackage{graphicx}
\usepackage{graphics}
\usepackage{url,hyperref}
\usepackage{xcolor}
\usepackage[margin=1.25in]{geometry}

\newtheorem{theorem}{Theorem}[section]

\newtheorem{corollary}[theorem]{Corollary}

\newtheorem{definition}[theorem]{Definition}

\newtheorem{lemma}[theorem]{Lemma}

\newtheorem{proposition}[theorem]{Proposition}
\newtheorem*{remark}{Remark}

\begin{document}

\title[The sub-Riemannian length spectrum]{The sub-Riemannian length spectrum\\
for screw motions of constant pitch\\
on flat and hyperbolic 3-manifolds}
\author{Marcos Salvai}
\thanks{This work was supported by Consejo Nacional de Investigaciones Cient%
	\'{\i}ficas y T\'ecnicas and Secretar\'{\i}a de Ciencia y T\'ecnica de la
	Universidad Nacional de C\'ordoba.}

\begin{abstract}
Let $M$ be an oriented three-dimensional Riemannian manifold of constant
sectional curvature ~$k=0,1,-1$ and let $SO\left( M\right) $ be its direct
orthonormal frame bundle (direct refers to positive orientation), which may
be thought of as the set of all positions of a small body in $M$. Given $%
\lambda \in \mathbb{R}$, there is a three-dimensional distribution $\mathcal{%
D}^{\lambda }$ on $SO\left( M\right) $ accounting for infinitesimal
rototranslations of constant pitch $\lambda $. When $\lambda \neq k^{2}$,
there is a canonical sub-Riemannian structure on $\mathcal{D}^{\lambda }$.
We present a geometric characterization of its geodesics, using a previous
Lie theoretical description.

For $k=0,-1$, we compute the
sub-Riemannian length spectrum of $\left( SO\left( M\right),\mathcal{D}%
^{\lambda }\right) $ in terms of the complex length spectrum of $M$ (given
by the lengths and the holonomies of the periodic geodesics) when $M$ has
positive injectivity radius. In particular, for two complex length
isospectral closed hyperbolic 3-manifolds (even if they are not isometric),
the associated sub-Riemannian metrics on their direct orthonormal bundles
are length isospectral.
\end{abstract}

\maketitle


\smallskip


\section{Introduction}

In this paper we work in the setting of sub-Riemannian geometry. In a
sub-Rie\-man\-nian manifold, any point is allowed to move only in certain
special directions and can reach any other point through an admissible
motion. There is a notion of length of admissible curves, which induces a
distance on the manifold. The locally shortest curves are called geodesics
and have been extensively studied.

The description of \textsl{periodic} sub-Riemannian geodesics, in
particular, the sub-Riem\-an\-nian length spectrum, has been addressed to a
lesser extent. Beyond the paradigmatic case of the standard sub-Riemannian
structure on odd dimensional spheres (see \cite{ChangMarkinaV, GodoyM}), it
has attracted interest only recently \cite{Corey, YCV} (see also \cite%
{Involve}, where the spherical case is revisited). We feel that further
concrete, natural examples could be welcome.

\bigskip

\noindent \textbf{Outline of the contents of the paper. }Next we present the
subject of the article in a sketchy, slightly informal manner. Let $M$ be a
complete oriented three-dimensional Riemannian manifold of constant
sectional curvature $k=0,1,-1$ and positive injectivity radius $\rho $.

Suppose we have a small extended body with reference state at a point $o\in
M $ (by this we mean that it is contained in an open ball $B$ centered at $o$
of radius smaller than $\rho $, and not in the image of a geodesic). Let $%
SO\left( M\right) $ be the set of all the direct orthonormal tangent
3-frames of $M$ (direct means positively oriented). The set of all positions
of the body in $M$ is in bijection with $SO\left( M\right) $, as follows.
Fix a direct orthonormal basis $\mathcal{B}_{o}$ of $T_{o}M$. Given $p\in M$
and a direct orthonormal basis $\mathcal{B}$ of $T_{p}M$, consider the
unique isometry $F$ defined on $B_{o}$ satisfying that $F\left( o\right) =p$
and whose differential map at $o$ sends $\mathcal{B}_{o}$ to $\mathcal{B} $.
We think of $F$ as transporting the body to a new position at $p$.

Given $\lambda \in \mathbb{R}$, we impose a \textbf{constraint on the motion
of the body}: At the infinitesimal level, it is allowed to translate in any
direction, only if it simultaneously rotates around that direction with
angular speed $\lambda $. By the identification above, this gives
distinguished curves of orthonormal 3-frames of $M$, which are said to be
admissible. It turns out that if $\lambda ^{2}\neq k$ (which we henceforth
assume), then any position can be attained from the original one by means of
a continuous piecewise admissible motion.

There is a canonical way of measuring the length of admissible curves,
giving a so called sub-Riemannian structure on $SO\left( M\right) $. In our
first main result, Theorem \ref{geodesica}, we describe explicitly all the
motions which are locally the shortest ones. They are called \textbf{%
sub-Riemannian geodesics}. We obtain in particular that \textbf{they project
to helices} with constant speed. We resort to the paper \cite{HMS}, where we
have studied a much more general situation, with a Lie theoretical approach.

It is natural to pose the question of which of the sub-Riemannian geodesics
are \textbf{periodic}. We answer it thoroughly. The first step is the fact
that if a geodesic is periodic, so is the projected helix in $M$. If a
periodic helix is neither a periodic geodesic nor a circle, then it has an
axis (which is a periodic helix). Suppose that the axis has length $\ell $
and \textbf{holonomy} $\theta $ (this is the amount of rotation a
transversal disc undergoes when it is parallel translated once around the
axis). If the helix \textbf{turns }$p$\textbf{\ times around the axis while
the latter traverses }$q$\textbf{\ times its period}, it is said to be of
type $\left( \ell +i\theta ,q,p\right) $. We comment that counting the turns
is a quite touchy issue due the holonomy.

We find the \textbf{closing condition} for a sub-Riemannian geodesic in $%
SO\left( M\right) $ projecting to a helix of type $\left( \ell +i\theta
,q,p\right) $ in terms of $\ell $, $\theta $, $q$ and $p$ (in between, we
must relate these numbers with the curvature and the torsion, in view of the
characterization of the sub-Riemannian geodesics); it holds only for certain
radii of the helices. This allows us to obtain our second main result,
Theorem \ref{unionSpectrum}, the set of all lengths of periodic
sub-Riemannian geodesics in terms of the lengths and holonomies of the
periodic geodesics of $M$.

\bigskip

We have taken a similar approach in \cite{CS} to describe the geodesics on
the unit tangent bundle of a complete oriented hyperbolic 3-manifold endowed
with the canonical (Sasaki) metric, and to compute its length spectrum when
the injectivity radius is positive; the article \cite{salvai} deals with the
two-dimensional case. The problems in the present paper are somehow more
demanding.

Perhaps the simplest case $\lambda =0$ (infinitesimal pure translations),
for $k\neq 0$, is the most relevant. It can be generalized in the obvious
manner to all dimensions of $M_{k}$ (not only three) with $k=1,-1$.
Sometimes it is called the $K+P$ problem. For surfaces, that is, $SO\left(
3\right) $ and $O_{o}\left( 1,2\right) $ acting on the two-sphere and the
hyperbolic plane, respectively, as well as the similar situation of $S^{3}$
covering $SO\left( 3\right) $, we have the simplest examples of invariant
sub-Riemannian geometry on simple Lie groups. They are presented in any book
on the field and have been vastly studied. See for instance \cite{Boscain,
Boscain2, Grong2011, BZ, Sachkov22}.

Finally, we comment briefly on topics that could be studied on the basis of
the results presented above. The first question is whether one can compute
the heat kernel and the sub-Laplace spectrum of the sub-Riemannian manifolds
considered in the article. For this, notice, on the one hand, that much is
known on the representation theory of the semisimple Lie groups $O_{o}\left(
1,3\right) $ and $SO\left( 4\right) $, and on the other hand, that the Popp
volume form is the Haar measure, by left invariance \cite{Agra}. Papers
related to this endeavor are, for example, \cite{AgraIntrinsic, Cardona}. In
Example 11.4 of the latter, putting $c=0$ we have our case $k=1$, $\lambda
=0 $; if the system is modified setting $Y^{c}$ and $Z^{c}$ instead of $%
Y^{-} $ and $Z^{-}$, we would have our case $k=1,$ $\lambda =c$ (in both
cases, up to a double covering). In a second instance, the quest for a trace
formula relating the Laplace and length spectra would arise. For this, the
multiplicities of the latter would be needed (see the comment after the
statement of Theorem \ref{unionSpectrum}).

In Section \ref{Section2} we introduce the sub-Riemannian manifolds we deal
with. In Section \ref{Section3} we describe the sub-Riemannian geodesics. In
Section \ref{Section4} we state the closing condition for sub-Riemannian
geodesics and present explicitly the length spectrum. Section \ref{Section5}
is devoted to the geometry of periodic helices, which will be used in the
proofs of the results, given in Section \ref{Section6}.

\section{The $\protect\lambda $-screw distribution on the orthonormal frame
bundle\label{Section2}}

Let $M_{k}$ be the three-dimensional space form of constant sectional
curvature $k=0,1,-1$, that is, $M_{0}=\mathbb{R}^{3}$, $M_{1}=S^{3}$ and $%
M_{-1}$ is hyperbolic space $H^{3}$. Let $G_{k}$ be the identity component
of the isometry group of $M_{k}$. In \cite{HMS} we have defined a left
invariant sub-Riemannian structure $\left( \mathcal{D}^{\lambda
},\left\langle ,\right\rangle \right) $ on $G_{k}$, for $\lambda \neq k^{2}$%
, as follows (for definitions and properties in sub-Riemannian geometry, we
refer to \cite{RM, Agra}).

For $k=\pm 1$, consider the inner product on $\mathbb{R}^{4}$ given by%
\begin{equation}
\left\langle x,y\right\rangle
_{k}=kx_{0}y_{0}+x_{1}y_{1}+x_{2}y_{2}+x_{3}y_{3}\text{.}
\label{InnerProduct}
\end{equation}%
If $M_{k}$ is presented as usual as the connected component of $e_{0}$ of $%
\left\{ x\in \mathbb{R}^{4}\mid \left\langle x,x\right\rangle _{k}=k\right\} 
$, then $G_{1}=SO\left( 4\right) $ and $G_{-1}=O_{o}\left( 1,3\right) $. We
identify $\mathbb{R}^{3}$ with $\left\{ x\in \mathbb{R}^{4}\mid
x_{0}=1\right\} $, and so we have 
\begin{equation*}
G_{0}=\left\{ \left( 
\begin{array}{cc}
1 & 0 \\ 
a & A%
\end{array}%
\right) \mid a\in \mathbb{R}^{3}\text{, }A\in SO\left( 3\right) \right\} 
\text{.}
\end{equation*}

For $k=0,1,-1$, let $\mathfrak{g}_{k}$ be the Lie algebra of $G_{k}$ and let 
$\mathfrak{g}_{k}=\mathfrak{p}_{k}\oplus \mathfrak{k}_{k}$ be the Cartan
decomposition associated with $e_{0}$, into infinitesimal translations of $%
M_{k}$ through $e_{0}$ and infinitesimal rotations around that point.
Denoting $\mathfrak{so}\left( 3\right) =\left\{ X\in \mathbb{R}^{3\times
3}\mid X^{T}=-X\right\} $, we have%
\begin{equation*}
\mathfrak{p}_{k}=\left\{ \left( 
\begin{array}{cc}
0 & -kx^{T} \\ 
x & 0%
\end{array}
\right) \mid x\in \mathbb{R}^{3}\right\} \text{\quad and \quad }\mathfrak{k}
_{k}=\left\{ \left( 
\begin{array}{cc}
0 & 0 \\ 
0 & X%
\end{array}
\right) \mid X\in \mathfrak{so}\left( 3\right) \right\} \text{.}
\end{equation*}%
For $x\in \mathbb{R}^{3}$ we set%
\begin{equation}
L_{x}\in \mathfrak{so}\left( 3\right) \text{,\qquad }L_{x}\left( y\right)
=x\times y\text{,}  \label{Lx}
\end{equation}%
for $y\in \mathbb{R}^{3}$, where $\times $ is the usual cross product on $%
\mathbb{R}^{3}$.

\begin{definition}
For $\lambda \in \mathbb{R}$, the $\lambda $-screw distribution on $G_{k}$
is the left invariant distribution $\mathcal{D}^{\lambda }$ given at the
identity by 
\begin{equation*}
\mathcal{D}_{I}^{\lambda }=\left\{ D^{\lambda }\left( x\right) :=\left( 
\begin{array}{cc}
0 & -kx^{T} \\ 
x & \lambda L_{x}%
\end{array}%
\right) \mid x\in \mathbb{R}^{3}\right\} \subset \mathfrak{g}_k\text{.}
\end{equation*}
\end{definition}

Thinking of the elements of $G_{k}$ as positions of a body in $M_{k}$ with
reference state at $e_{0}$, then, according to the constraints induced by $%
\mathcal{D}^{\lambda }$, at the infinitesimal level, the body can be
translated $c$ units of length along any direction only if at the same time
it rotates around that direction through the angle $\lambda c$. Submanifolds of $G_k$ 
related with rototranslations have been studied in \cite{SalvaiVorticity}.

From now on we assume that $\lambda \neq k^{2}$. This is the necessary and
sufficient condition we found in \cite{HMS} for the distribution $\mathcal{D}%
^{\lambda }$ to be bracket generating (that is, the vector fields in $%
\mathcal{D}^{\lambda }$ generate the Lie algebra of vector fields of $G_{k}$%
). In this case, the Chow-Rashevsky Theorem assures that for each pair of
points in $G_{k}$ there exists a continuous piecewise horizontal curve
joining them (a smooth curve $\sigma $ in $G_{k}$ is said to be horizontal
or admissible if $\sigma ^{\prime }\left( t\right) \in \mathcal{D}_{\sigma
\left( t\right) }^{\lambda }$ for all $t$).

We consider the left invariant sub-Riemannian structure on $\left( G_{k},%
\mathcal{D}^{\lambda }\right) $ determined by 
\begin{equation}
\left\Vert D^{\lambda }\left( x\right) \right\Vert =\left\Vert x\right\Vert\text{,}
  \label{s-Rmetric}
\end{equation}%
for $x\in \mathbb{R}^{3}$. It is called the $\lambda $-screw sub-Riemannian structure on $G_{k}$. The
group $SO\left( 3\right) $ acts on the right on $G_{k}$ by sub-Riemannian
isometries. Thus, $G_{k}\times SO\left( 3\right) $ acts transitively on
orthonormal bases of $\mathcal{D}^{\lambda }$ and so, up to coverings, $%
\left( G_{k},\mathcal{D}^{\lambda }\right) $ is a sub-Riemannian model space
as in \cite{GrongModel}.

A diffeomorphism between oriented manifolds is said to be direct if it
preserves the orientation. Let $N$ be an oriented three-dimensional
Riemannian manifold and let $SO\left( N\right) $ be the bundle of positively
oriented orthonormal frames of $N$, i.e., 
\begin{equation*}
SO\left( N\right) =\left\{ \left( p,b\right) \mid p\in N\text{ and }b:%
\mathbb{R}^{3}\rightarrow T_{p}N\text{ is a direct linear isometry}\right\} 
\text{,}
\end{equation*}%
which is a principal fiber bundle over $N$ with typical fiber $SO\left(
3\right) $ (sometimes we will omit the foot points). When $N=M_{k}$, the
natural identification $\mathbb{R}^{3}\equiv \left\{ x\in \mathbb{R}^{4}\mid
x_{0}=0\right\}$ $=T_{e_{0}}M_{k}$ allows us to define the following
diffeomorphism between $G_{k}$ and $SO\left( M_{k}\right) $:%
\begin{equation}
\Phi :G_{k}\rightarrow SO\left( M_{k}\right) \text{,\qquad }\Phi \left(
g\right) =\left( g\left( e_{0}\right) ,dg_{e_{0}}\right) \text{.}
\label{Phi}
\end{equation}

We consider on $SO\left( M_{k}\right) $ the sub-Riemannian structure induced
from $\mathcal{D}^{\lambda }$ via $\Phi $. By abuse of notation, we also
call it $\mathcal{D}^{\lambda }$. Notice that $\Phi $ is equivariant by the
natural actions of $G_{k}\times SO\left( 3\right) $.

Now let $M$ be a complete oriented three-dimensional Riemannian manifold of
constant sectional curvature $k=0,1,-1$. The Riemannian universal covering
of $M$ is isometric to $M_{k}$. The fundamental group $\Gamma $ of $M$ acts
freely and properly discontinuously on $M_{k}$, and we may identify $M$ with 
$\Gamma \backslash M_{k}$. Similarly, $SO\left( M\right) $ may be identified
with $\Gamma \backslash SO\left( M_{k}\right) $. Also, the distribution $%
\mathcal{D}^{\lambda }$ and its sub-Riemannian structure descend well to $%
SO\left( M\right) $. By abuse of notation, we also call $\mathcal{D}%
^{\lambda }$ the distribution on the quotient.

When describing sub-Riemannian geodesics of $\left( SO\left( M\right) ,%
\mathcal{D}^{\lambda }\right) $ we do not need to assume that the
injectivity radius is positive (as opposed to the case when we compute the
length spectrum). The moving body interpretation in the introduction is
still valid if we admit self-intersections of the body as it goes farther
and farther into a cusp.

Regarding closed three-dimensional flat manifolds, we would like to mention 
\cite{CR}. In the present paper we are interested in the six chiral
(oriented) platycosms.

\section{Sub-Riemannian geodesics\label{Section3}}

Let $M$ be a complete oriented three-dimensional Riemannian manifold of
constant sectional curvature $k=0,1,-1$, and consider on $\left( SO\left(
M\right) ,\mathcal{D}^{\lambda }\right) $ the sub-Riemannian structure
defined in the previous section, for $\lambda \neq k^{2}$.

The length of a horizontal curve $\sigma :\left[ a,b\right] \rightarrow
SO\left( M\right) $ is defined by $\operatorname{length}\left( \sigma \right)
=\int_{a}^{b}\left\Vert \sigma ^{\prime }\left( t\right) \right\Vert ~dt$.
For arbitrary $b_{1},b_{2}\in SO\left( M\right) $,%
\begin{equation*}
d\left( b_{1},b_{2}\right) =\inf \left\{ \operatorname{length}\left( \sigma \right)
\mid \sigma \text{ is a piecewise horizontal curve joining }b_{1}\text{ with 
}b_{2}\right\}
\end{equation*}%
defines a \textbf{distance} on $SO\left( M\right) $. A constant speed
horizontal curve $\gamma $ in $SO\left( M\right) $ is said to be a \textbf{%
sub-Riemannian geodesic} if it minimizes the distance locally, that is, if $%
a<b$ in the domain of $\gamma $ are close enough, then the length of $\left.
\gamma \right\vert _{\left[ a,b\right] }$ equals the distance between $%
\gamma \left( a\right) $ and $\gamma \left( b\right) $. For the sake of
simplicity, we sometimes call them just geodesics.

In \cite{HMS} we have given a Lie theoretic description of all
sub-Riemannian geodesics of $\left( G_{k},\mathcal{D}^{\lambda }\right) $.
Now we present a geometric characterization of them, more suitable in light
of the identification (\ref{Phi}), that can be also stated for the quotient $%
SO\left( M\right) \equiv \Gamma \backslash G_{k}$. In order to do that, we
recall properties of helices and introduce some notation.

A \textbf{helix} $h:\mathbb{R}\rightarrow M$ is a smooth curve with constant
speed, constant curvature $\kappa $ and constant torsion $\tau $ (in the
case that $\kappa \neq 0$).

Now assume that $h$ has unit speed. For $\kappa >0$, let $T,N,B:\mathbb{R}%
\rightarrow T^{1}M$ be the velocity of $h$ and the normal and binormal
vector fields along $h$, respectively. The \textbf{Frenet frame} of $h$ is
the curve $F:\mathbb{R}\rightarrow SO\left( M\right) $ defined by $F\left(
t\right) \left( e_{i}\right) =T\left( t\right) $, $N\left( t\right) $ and $%
B\left( t\right) $, for $i=1,2,3$, respectively.

Suppose that $\kappa =0$ (that is, $h$ is a geodesic). In this case we adopt
the convention that the torsion is zero. Let $N,B$ be parallel vector fields
along $h$ such that $\left\{ h^{\prime }\left( t\right) =T\left( t\right)
,N\left( t\right) ,B\left( t\right) \right\} $ is a direct orthonormal basis
of $T_{h\left( t\right) }M$ for all $t$. The curve $F:\mathbb{R}\rightarrow
SO\left( M\right) $ defined by%
\begin{equation}
F\left( t\right) \left( e_{1}\right) =T\left( t\right) \text{,\quad }F\left(
t\right) \left( e_{2}\right) =N\left( t\right) \quad \text{and\quad }F\left(
t\right) \left( e_{3}\right) =B\left( t\right)  \label{FrenetGeo}
\end{equation}%
is said to be a \textbf{Frenet frame} of $h$. We have then that the Frenet
equations $\frac{DT}{dt}=\kappa N$, $\frac{DN}{dt}=-\kappa T+\tau B$ and $%
\frac{DB}{dt}=-\tau N$ hold for all helices, including geodesics.

Given $v\in \mathbb{R}^{3}$ and $t\in \mathbb{R}$, we denote%
\begin{equation*}
\operatorname{Rot}\left( v,t\right) =\exp _{SO\left( 3\right) }\left( tL_{v}\right)
\end{equation*}%
(see (\ref{Lx})). We have that $R\left( 0,t\right) $ is de the identity for
all $t$ and, if $v\neq 0$, then $\operatorname{Rot}\left( v,t\right) $ is the
rotation of $\mathbb{R}^{3}$ with axis $\mathbb{R}v$ which rotates the plane
orthogonal to $v$ through the angle $\left\Vert v\right\Vert t$, more
precisely, satisfying $\operatorname{Rot}\left( v,t\right) v=v$ and%
\begin{equation*}
\operatorname{Rot}\left( v,t\right) x=\cos \left( \left\Vert v\right\Vert t\right)
x+\sin \left( \left\Vert v\right\Vert t\right) \left( v\times x\right)
/\left\Vert v\right\Vert
\end{equation*}%
for $x\perp v$. There is an analogous definition for an oriented
three-dimensional vector space with an inner product, with the induced cross
product.

Now, we are in a position to characterize geometrically all geodesics of our
sub-Riemannian manifold.

\begin{theorem}
\label{geodesica}Let $M$ be a complete oriented three-dimensional Riemannian
manifold of constant sectional curvature $k=0,1,-1$, and let $\lambda \neq
k^{2}$. Let $h$ be a unit speed helix in $M$ with curvature $\kappa $ and
torsion $\tau $. Let $F$ be a Frenet frame for $h$ and $O\in SO\left(
3\right) $. Then%
\begin{equation}
\gamma \left( t\right) =\left( h\left( t\right) ,F\left( t\right) \operatorname{Rot}%
\left( \left( \lambda -\tau ,0,-\kappa \right) ,t\right) O\right)
\label{s-Rgeodesic}
\end{equation}%
is a sub-Riemannian geodesic of $\left( SO\left( M\right) ,\mathcal{D}%
^{\lambda }\right) $, as well as the curve $\sigma $ given by $\sigma \left(
t\right) =\gamma \left( dt\right) $, for $d\geq 0$. Conversely, any
sub-Riemannian geodesic has this form.
\end{theorem}

\begin{remark}
\label{Remarkab}\emph{a)} Writing $\gamma \left( t\right) =\left( h\left(
t\right) ,b\left( t\right) \right) $, we may think of $b\left( t\right) $ as
combining the rotation of the Frenet frame according to the Frenet
equations, with the rotation around the axis spanned by $\left( \lambda
-\tau \right) T\left( t\right) -\kappa B\left( t\right) $ \emph{(}which is
constant with respect to the Frenet frame\emph{)}, with angular speed equal
to $\sqrt{\kappa ^{2}+\left( \lambda -\tau \right) ^{2}}$. When $\lambda =0$%
, this number is called in \emph{\cite{gluck}} the writhe of the helix.

\smallskip

\emph{b)} In the particular case when $\tau =\lambda $, it is not difficult
to visualize that the geometric interpretation of the constraint given by $%
\mathcal{D}^{\lambda }$ is consistent with the meaning of the torsion of the
helix $h$. Indeed, $F\left( t\right) \operatorname{Rot}\left( \left( 0,0,-\kappa
\right) ,t\right) F\left( t\right) ^{-1}$ is a rotation of $T_{h\left(
t\right) }M$ around $B\left( t\right) $, that is, it preserves the
osculating plane, while the latter is rotating along and around $h$ with
angular speed $\tau =\lambda $, by the Frenet equation $B^{\prime }=-\tau N$.
\end{remark}

\section{The sub-Riemannian length spectrum\label{Section4}}

\begin{definition}
Let $N$ be a smooth manifold. A smooth curve $\sigma :\mathbb{R}\rightarrow
N $ is said to be \textbf{periodic} if it is not constant and there is a
positive number $t$ such that $\sigma \left( t+s\right) =\sigma \left(
s\right) $ for all $s$. It is not difficult to show that there exists the
smallest of such numbers $t$, say, $t_{o}$, which is called the \textbf{%
period} of $\sigma $. If $N$ is \emph{(}sub-\emph{)}Riemannian and $\sigma $
is horizontal, by the \textbf{length} of $\sigma $ we understand the length
of $\sigma $ restricted to the interval $\left[ 0,t_{o}\right] $. The 
\textbf{length spectrum} of $N$ is the collection of all the lengths of its
periodic geodesics.
\end{definition}

Note that the term length spectrum has different variants in the literature.
We adopted the one used in \cite{Corey} (in particular, we do not consider
multiplicities; see the paragraph below Theorem \ref{unionSpectrum}). It is
sometimes called the \textsl{primitive} length spectrum, but we do not need
the extra qualifier, since the length of $\left. \gamma \right\vert _{\left[
0,nt_{o}\right] }$, for $t_{o}$ the period of $\gamma $, will be taken into
account only for $n=1$.

\medskip

In our context, geodesics in $M_{k}$ and its quotients are mostly axes of
helices (which appear as projections of sub-Riemannian geodesics of the $%
\lambda $-screw structure) and will be often denoted by $\alpha \,$. We
reserve the letter $\gamma $ for sub-Riemannian geodesics of the manifold $SO\left( M,%
\mathcal{D}^{\lambda }\right) $.

\begin{definition}
\label{holonomy}Let $\alpha $ be a periodic geodesic of an oriented
three-dimensional Riemannian manifold and let $t_{o}$ be its period. The 
\textbf{holonomy} of $\alpha $ is the unique $\theta \in \lbrack 0,2\pi )$
such that $\mathcal{T}_{0,t_{o}}=$ $\operatorname{Rot}\left( \alpha ^{\prime }\left(
0\right) ,\theta \right) $, where $\mathcal{T}_{s,t}$ denotes the parallel
transport along $\alpha $ from $s$ to $t$. The \textbf{complex length} of $%
\alpha $ is the complex number $\ell +i\theta $, where $\ell $ is the length
of $\alpha $. The \textbf{complex length spectrum} of the manifold is the
collection of all complex lengths of periodic geodesics.
\end{definition}

As far as we know, the notion of complex length was introduced in \cite%
{mayer87} and is since a central concept in the geometry of 3-manifolds. In 
\cite{reid} it begins to be exploited in relation to isospectrality. We have
worked with the higher dimensional version in \cite{salvai2}.

\bigskip%

The following proposition is a consequence of the description of the
sub-Riemannian geodesics in Theorem \ref{geodesica}. We prove it in
Subsection \ref{ProofsSRLS}.

\begin{proposition}
\label{fundamentalProp}Let $M$ be a complete oriented three-dimensional
Riemannian manifold of constant sectional curvature $k=0,1,-1$, and let ${%
\lambda }\neq k^{2}$.

\smallskip

\emph{a)} A periodic sub-Riemannian geodesic of $\left( SO\left( M\right) ,%
\mathcal{D}^{\lambda }\right) $ projects to a periodic helix.

\smallskip

\emph{b)} Let $\alpha $ be a unit speed periodic geodesic of $M$ with
complex length $\ell +i\theta $. Then a sub-Riemannian geodesic $\gamma $ of 
$\left( SO\left( M\right) ,\mathcal{D}^{\lambda }\right) $ projecting to $%
\alpha $ is periodic if and only if there exist coprime integers $%
n,m^{\prime }$, with $n>0$, such that $n\left( \lambda \ell +\theta \right)
=2m^{\prime }\pi $ \emph{(}by convention, we take $n=1$ if $\lambda \ell
+\theta =0$\emph{)}. In this case, $\operatorname{length}\left( \gamma \right)
=n\ell $.

\smallskip

\emph{c)} Let $h$ be a periodic unit speed helix on $M$ with length $L$,
curvature $\kappa >0$ and torsion $\tau $. Then a sub-Riemannian geodesic $%
\gamma $ of $\left( SO\left( M\right) ,\mathcal{D}^{\lambda }\right) $
projecting to $h$ is periodic if and only if there exist coprime positive
integers $m,n$ such that 
\begin{equation}
nL\sqrt{\kappa ^{2}+\left( \lambda -\tau \right) ^{2}}=2m\pi \text{.}
\label{FundamentalEqShort}
\end{equation}%
In this case,%
\begin{equation}
\operatorname{length}\left( \gamma \right) =nL\text{.}  \label{lengthNL}
\end{equation}
\end{proposition}

\begin{remark}
We have found the following curious fact: If we adopt the convention that a
periodic geodesic with complex length $\ell +i\theta $ has torsion equal to $%
-\theta /\ell $ \emph{(}and not zero, although it is a geodesic\emph{)},
then \emph{(b)} is a particular case of \emph{(c)}, with $m^{\prime }=\operatorname{%
sign}\left( \lambda \ell +\theta \right) m$.

\smallskip

Regarding \emph{(c)}, for the periodic sub-Riemannian geodesic $\left(
h,b\right) $, when the periodic helix $h$ is is traversed $n$ times, $b$
rotates $m$ times with angular speed $\sqrt{\kappa ^{2}+\left( \lambda -\tau
\right) ^{2}}$ around the axis determined by $\left( \lambda -\tau \right)
T\left( t\right) -\kappa B\left( t\right) $, that is constant with respect
to the Frenet frame, while the latter, in turn, rotates according to the
Frenet equations. Notice that $h$ and its Frenet frame have the same period
if $\kappa >0$.
\end{remark}

\subsection{The sub-Riemannian length spectrum of $\left( SO\left(
M_{k}\right) ,\mathcal{D}^{\protect\lambda }\right) $\nopunct}

\mbox{ }

\smallskip

For $k=0,1,-1$ we introduce the notation%
\begin{equation}
\sin _{1}r=\sin r\text{,\quad }\sin _{0}r=r\text{,\quad}\sin _{-1}r=\sinh r%
\text{,\quad and \quad}\cos _{\kappa }r=\sin _{\kappa }^{\prime }r
\label{sink}
\end{equation}%
(in particular, $\cos _{k}^{2}+k\sin _{k}^{2}=1$). Similarly, $\cot
_{k}=\cos _{k}/\sin _{k}$.

We recall that for $k=0,1,-1$, circles in $M_{k}$ are characterized as the
helices with positive curvature (that is, not geodesics) contained in a
totally geodesic surface (that is, with zero torsion) and satisfying the
extra condition $\kappa >1$ for $k=-1$. If $k=1$ we require the radius of
the circle to be smaller than $\pi $.

\begin{proposition}
\label{PropoBasic}Let $k=0,1,-1$ and $\lambda \in \mathbb{R}$ with $\lambda
\neq k^{2}$.

\smallskip

\emph{a)} Periodic sub-Riemannian geodesics of $\left( SO\left( M_{k}\right)
,\mathcal{D}^{\lambda }\right) $ project to circles.

\smallskip

\emph{b)} Let $h$ be a unit speed circle of radius $r$ in $M_{k}$. Then a
periodic sub-Riemannian geodesic $\gamma $ in $\left( SO\left( M_{k}\right) ,%
\mathcal{D}^{\lambda }\right) $ projects to $h$ if and only there exist
coprime positive integers $m,n$ such that%
\begin{equation}
\sin _{k}^{2}r=\frac{m^{2}-n^{2}}{n^{2}\left( \lambda ^{2}-k\right) }\text{,}
\label{sinRadius}
\end{equation}%
where $m>n$, except for $k=1$ and $\lambda ^{2}<1$, when we have $m<n$.
\end{proposition}

It is well-known that the length of a circle of radius $r$ in $M_{k}$ is
equal to $2\pi \sin _{k}r$. Hence, (\ref{sinRadius}) together with (\ref%
{lengthNL}) yield straight away the following corollary.

\begin{corollary}
\label{lengthSsc}For $k=0,1,-1$ and $\lambda \neq k^{2}$, the length
spectrum of $\left( SO\left( M_{k}\right) ,\mathcal{D}^{\lambda }\right) $
is 
\begin{equation}
\mathcal{L}:=\left\{ 2\pi \sqrt{\frac{m^{2}-n^{2}}{\lambda ^{2}-k}}\mid 
\begin{tabular}{l}
$m,n\text{ are coprime positive integers }$ \\ 
$\text{ such that the radicand is positive}$%
\end{tabular}%
\right\} \text{.}  \label{LSarriba}
\end{equation}

In particular, for any $k=0,1,-1$, the canonical sub-Riemannian metrics on $%
\mathcal{D}^{\lambda }$ and $\mathcal{D}^{\lambda ^{\prime }}$ are not
isometric if $\lambda \neq \lambda ^{\prime }$ \emph{(}and distinct from $%
k^{2}$\emph{)}.
\end{corollary}

We comment on the number $\sin _{k}r$, which appears many times in the
paper. For $k=0$ it is just the radius $r$. For $k=-1,1$, it is the distance
from the helix to its axis, measured on the horosphere perpendicular to the
latter (called the horospherical radius of the helix), or measured according
to the extrinsic metric of $S^{3}$ in $\mathbb{R}^{4}$, respectively.

\subsection{The sub-Riemannian length spectrum of $\left( SO\left( \Gamma
\backslash M_{k}\right) ,\mathcal{D}^{\protect\lambda }\right) $\nopunct}

\mbox{ }

\smallskip

Let $M$ be a complete oriented three-dimensional Riemannian manifold with
constant sectional curvature $k=0,-1$ (from now on in this section we
exclude the spherical case). We want to compute the length spectrum of $%
\left( SO\left( M\right) ,\mathcal{D}^{\lambda }\right) $ in terms of the
complex length spectrum of $M$, when the manifold has positive injectivity
radius (for instance, when $k=0$ or $M$ is closed).

We know from Proposition \ref{fundamentalProp} (a) that sub-Riemannian
geodesics of $\left( SO\left( M\right) ,\mathcal{D}^{\lambda }\right) $
project to helices. We have dealt with geodesics and circles in Propositions %
\ref{fundamentalProp} (b) and \ref{PropoBasic} (b), respectively. The
remaining helices, except for horocycles, have axes. Next we introduce them:

Let $\alpha $ be a geodesic in $M$ with speed $c>0$. In Definition \ref%
{radiusEtc} below we define the notion of a helix in $M_{k}$ with \textbf{%
axis} $\alpha $, \textbf{radius} $r>0$ and \textbf{angular speed} $\mu \in 
\mathbb{R}$ (the concepts that the reader has surely in mind). If $k=0$, we
require $\mu \neq 0$, to exclude geodesics of flat manifolds.

\begin{proposition}
\label{helixKappaTau}Let $k=0,-1$. A helix as above has unit speed if and
only if%
\begin{equation}
c^{2}\left( \cos _{k}^{2}r+\mu ^{2}\sin _{k}^{2}r\right) =1\text{.}
\label{UnitSpeed}
\end{equation}%
In this case, $c^{2}<1$ and the relationship establishes a bijection between 
$c\in \left( 0,1\right) $ and $r\in \left( 0,\infty \right) $. The curvature
and the torsion of $h$ are given by%
\begin{equation}
\kappa ^{2}=\left( c^{2}\mu ^{2}-k\right) \left( 1-c^{2}\right) \text{\qquad
and \qquad }\tau =c^{2}\mu \text{,}  \label{CurvAndTorsion}
\end{equation}%
respectively.
\end{proposition}

The next proposition asserts that the axis of a periodic helix is periodic.
Moreover, it presents the notion of a helix turning $q$ times around its
axis while this runs $p$ times its period.

\begin{proposition}
\label{PeriodicHelix}Let $k=0,-1$. Let $h$ be a unit speed helix in $M$ with
angular speed $\mu $, and axis $\alpha $ of speed $c>0$. If $h$ is periodic,
say, of length $L$, then $\alpha $ is periodic, say, of complex length $\ell
+i\theta $, and there exist unique $q\in \mathbb{N}$ and $p\in \mathbb{Z}$
with $\left( p,q\right) =1$ such that%
\begin{equation}
cL=q\ell \text{\qquad and \qquad }\mu \ell =2\pi p/q-\theta \text{.}
\label{DisplayPer}
\end{equation}%
Conversely, suppose that $\alpha $ is periodic, say, of complex length $\ell
+i\theta $ and axis of speed $c>0$. Given $q\in \mathbb{N}$, $p\in \mathbb{Z}
$, $\left( p,q\right) =1$, there exists a periodic unit speed helix $h$ with
axis $\alpha $, whose length $L$ and angular speed $\mu $ are given by \emph{%
(\ref{DisplayPer})}.
\end{proposition}

\begin{definition}
A periodic helix as above is said to be of type $\left( \ell +i\theta
,q,p\right) $.
\end{definition}

\begin{remark}
Defining the number of turns is a delicate issue because of the holonomy of
the axis. For instance, if $q=1$ and $p=0$, then the torsion of $h$ is not
necessarily zero, that is, a helix turning zero times around the axis is not
necessarily contained in a totally geodesic surface containing the axis.
\end{remark}

\medskip

\noindent \textbf{Goal. }Our goal is to obtain the lengths of all periodic
sub-Riemannian geodesics $\gamma $ of $\left( SO\left( M\right) ,\mathcal{D}%
^{\lambda }\right) $ projecting to a helix $h$ of type $\left( \ell +i\theta
,q,p\right) $, only in terms of these parameters.

We have sketched the steps to achieve this in the introduction. We return to
them, providing the details.

By (\ref{CurvAndTorsion}) and the first equation in (\ref{DisplayPer}) we
have that $\kappa ,\tau $ and $L$ can be written in terms of $q$, $\ell $, $%
c $ and $\mu $, which, in their turn, depend only on $\left( \ell +i\theta
,q,p\right) $ and $r$, by the unit speed condition (\ref{UnitSpeed}) and the
second equation in (\ref{DisplayPer}).

We show these dependencies by expressing Equation (\ref{FundamentalEqShort})
as

\begin{equation}
nL_{\ell ,\theta ,q,p}\left( r\right) \sqrt{\left( \kappa _{\ell ,\theta
,q,p}\left( r\right) \right) ^{2}+\left( \lambda -\tau _{\ell ,\theta
,q,p}\left( r\right) \right) ^{2}}=2\pi m\text{,}  \label{bigEq}
\end{equation}%
which, more precisely, by Proposition \ref{helixKappaTau}, translates into%
\begin{equation}
n\frac{q\ell }{c}\sqrt{\left( c^{2}\mu ^{2}-k\right) \left( 1-c^{2}\right)
+\left( \lambda -c^{2}\mu \right) ^{2}}=2\pi m\text{,}  \label{Eq0}
\end{equation}%
where%
\begin{equation}
c^{2}=\frac{1}{\cos _{k}^{2}r+\mu ^{2}\sin _{k}^{2}r}\text{ \qquad and
\qquad }\mu =\frac{2\pi p/q-\theta }{\ell }\text{.}  \label{Eq2}
\end{equation}%
It is not difficult to show that Equation (\ref{Eq0}) is equivalent to%
\begin{equation}
\left( \left( \frac{2\pi m}{q\ell n}\right) ^{2}+\left( \lambda
^{2}-k\right) -\left( \mu -\lambda \right) ^{2}\right) c^{2}=\lambda ^{2}-k%
\text{.}  \label{EqN}
\end{equation}

The computations above are the basis of the proof of the following result,
which gives the lengths of all periodic sub-Riemannian geodesics of $%
SO\left( M\right) $ projecting to periodic helices of type $\left( \ell
+i\theta ,q,p\right) $.

\begin{theorem}
\label{TeoFormulaLarga}Let $M\ $be a complete oriented three-dimensional
Riemannian manifold of constant sectional curvature $k=0,-1$. Let $\lambda
\in \mathbb{R}$, with $\lambda \neq 0$ if $k=0$. Let $h$ be a periodic unit
speed helix in $M$ of type $\left( \ell +i\theta ,q,p\right) $ \emph{(}in
particular, the angular speed is $\mu =\left( 2\pi p/q-\theta \right) /\ell $%
\emph{)} and radius $r$.

Then $h$ is the projection of a periodic sub-Riemannian geodesic $\gamma $
of $SO\left( M,\mathcal{D}^{\lambda }\right) $ if and only if $r$ is the 
\emph{(}unique\emph{)} positive solution of the a system of equations \emph{(%
\ref{EqN})} and \emph{(\ref{Eq2})} for some positive coprime integers $m,n$
with%
\begin{equation}
2\pi m>nq\ell \left\vert \mu -\lambda \right\vert \text{,}
\label{conditionMN}
\end{equation}%
that is, 
\begin{equation}
r=\arcsin _{k}\left( \frac{4\pi ^{2}m^{2}-n^{2}\left( 2\pi p-q\theta
-\lambda \ell q\right) ^{2}}{n^{2}\left( \lambda ^{2}-k\right) \left( \left(
2\pi p-q\theta \right) ^{2}-k\ell ^{2}q^{2}\right) }\right) \text{.}
\label{radiusBig}
\end{equation}

In this case, we have that 
\begin{equation}
\operatorname{length}\left( \gamma \right) =\sqrt{\frac{4\pi ^{2}m^{2}+n^{2}\left(
\ell ^{2}\left( \lambda ^{2}-k\right) q^{2}-\left( 2\pi p-q\left( \theta
+\lambda \ell \right) \right) ^{2}\right) }{\lambda ^{2}-k}}\text{.}
\label{lengthBig}
\end{equation}
\end{theorem}

Notice that $r$ and the length of $\gamma $ are presented exclusively in
terms of $\ell +i\theta ,q,p,n,m$.

\medskip

Now we are ready to detail the sub-Riemannian length spectrum of $\left(
SO\left( M\right) ,\mathcal{D}^{\lambda }\right) $\ in terms of the complex
length spectrum of $M$, under the additional hypothesis that the injectivity
radius is positive.

Let $C\mathcal{L}\left( M\right) $ be the complex length spectrum of $M$ and
let $\ell +i\theta \in C\mathcal{L}$.

Let $\mathcal{L}\left( \ell +i\theta \right) $ be the set of all lengths of
sub-Riemannian geodesics in $SO\left( M,\mathcal{D}^{\lambda }\right) $\
projecting to periodic geodesics in $M$ with complex length $\ell +i\theta $%
, By Proposition \ref{fundamentalProp} (b) it is equal to%
\begin{equation*}
\mathcal{L}\left( \ell +i\theta \right) =\left\{ n\ell \mid 
\begin{tabular}{l}
$\lambda \ell +\theta =2\pi \tfrac{m}{n}~\text{for some coprime}$ \\ 
$\text{integers }n,m,\text{with }n>0$%
\end{tabular}%
\right\} \text{.}
\end{equation*}

Now, let $q,p$ coprime integers with $q>0$ (we establish the convention that 
$q=1$ if $p=0$). For any pair of coprime numbers $m,n\in \mathbb{N}$
satisfying the condition in (\ref{conditionMN}), or equivalently,%
\begin{equation}
4m^{2}\pi ^{2}>n^{2}\left( 2\pi p-\left( \theta +\lambda \ell \right)
q\right) ^{2}\text{,}  \label{EqMu}
\end{equation}%
call $l\left( \ell +i\theta ,q,p;n,m\right) $ the expression in (\ref%
{lengthBig}) and let%
\begin{equation*}
\mathcal{L}\left( \ell +i\theta ,q,p\right) =\left\{ l\left( \ell +i\theta
,q,p;n,m\right) \mid 
\begin{tabular}{l}
$n,m\in \mathbb{N}~\text{are coprime }$and \\ 
satisfy Equation (\ref{EqMu})%
\end{tabular}%
\right\} \text{.}
\end{equation*}%
By Theorem \ref{TeoFormulaLarga}, this is the set of all lengths of periodic
sub-Riemannian geodesics projecting to helices of type $\left( \ell +i\theta
,q,p\right) $.

The union of $\mathcal{L}$ in (\ref{LSarriba}) with the two sets above turns
out to be the whole sub-Riemannian length spectrum, provided that $M$ has
positive injectivity radius:

\begin{theorem}
\label{unionSpectrum}Let $M\ $be a complete oriented three-dimensional
Riemannian manifold of constant sectional curvature $k=0,-1$ with positive
injectivity radius. Let $\lambda \in \mathbb{R}$, with $\lambda \neq 0$ if $%
k=0$. Then the length spectrum of $SO\left( M,\mathcal{D}^{\lambda }\right) $
is given by 
\begin{equation*}
\mathcal{L}\bigcup \left( \bigcup_{\ell +i\theta \in \mathfrak{CL}\left(
M\right) }\mathcal{L}\left( \ell +i\theta \right) \right) \bigcup \left(
\bigcup_{\ell +i\theta \in \mathfrak{CL}\left( M\right) }\bigcup_{\mathcal{Q}%
}\mathcal{L}\left( \ell +i\theta ,q,p\right) \right) \text{,}
\end{equation*}%
where $\mathcal{L}$ is as in \emph{(\ref{LSarriba})} and $\mathcal{Q=}%
\left\{ \left( q,p\right) \in \mathbb{N}\times \mathbb{Z}\mid \left(
q,p\right) =1\right\} $.
\end{theorem}

As in \cite{Corey}, we do not consider multiplicities of the lengths. In our
setting, the difficulty to obtain them comes from the fact that the typical
fiber of the bundle $SO\left( M\right) \rightarrow M$, which is $SO\left(
3\right) $, is not simply connected (it has fundamental group $\mathbb{Z}%
_{2} $): Let $t\mapsto \gamma \left( t\right) =\left( h\left( t\right)
,b\left( t\right) \right) $ be a periodic sub-Riemannian geodesic of $\left(
SO\left( M\right) ,\mathcal{D}^{\lambda }\right) $ projecting to a periodic
helix of type $\left( \ell +i\theta ,q,p\right) $ and axis $\alpha $. Using $%
n$ as in Theorem \ref{TeoFormulaLarga} and that the helix is free homotopic
to $\left. \alpha \right\vert _{\left[ 0,qc\ell \right] }$, we would be able
to determine the free homotopy class of $\gamma $ as long as we could keep
track of the cumulative spinning of $b\left( t\right) $ in the period of $%
\gamma $, but we do not know how to do so (recall from Remark \ref{Remarkab}
(a) that $b\left( t\right) $ combines two rotations).

\begin{corollary}
If two closed oriented three-dimensional hyperbolic manifolds are complex
length isospectral, then their bundles of direct orthonormal frames, with
the corresponding $\lambda $-screw sub-Riemannian structures, are length
isospectral.
\end{corollary}

Note that there exist closed hyperbolic 3-manifolds with equal complex
length spectra but different volumes (see for instance \cite{ReidNuevo}), in
particular, not isometric.

\section{Helices\label{Section5}}

In order to prove the results we need some properties of helices. In
particular, we must relate the speed of the axis, the angular speed and the
radius with the curvature and the torsion. We also study closing conditions
in terms of the complex length of the axis. For $p\in M_{k}$, let $\operatorname{Exp}%
_{p}:T_{p}M_{k}\rightarrow M_{k}$ be the geodesic exponential map of $M_{k}$
associated with $p$.

\begin{definition}
Let $k=0,1,-1$. The helix in $M_{k}$ in standard position with radius $r>0$,
angular speed $\mu \in \mathbb{R}$ $\mathcal{(}$with $\mu \neq 0$ if $k=0%
\mathcal{)}$, and axis with speed $c>0$ is the curve%
\begin{equation*}
H\left( t\right) =\operatorname{Exp}_{A\left( t\right) }\left( r\left( \cos \left(
c\mu t\right) E_{2}\left( t\right) +\sin \left( c\mu t\right) E_{3}\left(
t\right) \right) \right) \text{,}
\end{equation*}%
where $A$ is the geodesic in $M_{k}$ with $A\left( 0\right) =e_{0}$ and
initial velocity $ce_{1}$, that is,%
\begin{equation*}
A\left( t\right) =\cos _{k}\left( ct\right) ~e_{0}+\sin _{k}\left( ct\right)
~e_{1}
\end{equation*}%
\emph{(}called the axis of $H$\emph{)}, and $E_{i}\left( t\right) $ is the
parallel transport of $e_{i}\in T_{e_{0}}M_{k}\equiv \mathbb{R}^{3}$ along $%
A $, between $0$ and $t$, for $i=1,2$.
\end{definition}

We have geodesic segments of length $r$, with one endpoint on the axis, and
orthogonal to it, rotating with angular speed $\mu $ around the axis, with
respect to arc length of the latter. The condition on $\mu $ in the
Euclidean case forbids straight lines in $\mathbb{R}^{3}$ as helices with
axes.

The helix $H$ can be also presented as the orbit of the point $p=\left( \cos
_{k}r,0,\sin _{k}r,0\right) $ under the action of the monoparametric
subgroup $t\mapsto \exp \left( tV\left( c,\mu \right) \right) $, with%
\begin{equation}
V\left( c,\mu \right) =c\left( 
\begin{array}{cc}
0 & -ke_{1}^{T} \\ 
e_{1} & L_{\mu e_{1}}%
\end{array}%
\right) \in \mathfrak{g}_{k}\text{.}  \label{VcMu}
\end{equation}%
Equivalently,%
\begin{equation*}
H\left( t\right) =\text{diag}\left( R_{k}\left( ct\right) ,R_{1}\left( c\mu
t\right) \right) \left( \cos _{k}r,0,\sin _{k}r,0\right) ^{T}\text{,}
\end{equation*}%
where%
\begin{equation*}
R_{k}\left( s\right) =\left( 
\begin{array}{cc}
\cos _{k}s & -k\sin _{k}s \\ 
\sin _{k}s & \cos _{k}s%
\end{array}%
\right) \text{.}
\end{equation*}%
Note that for $k=1,-1$, $R_{k}\left( s\right) $ is the standard or the
hyperbolic rotation through the angle $s$, respectively (recall from (\ref%
{sink}) the definition of $\sin _{k}$ and $\cos _{k}$).

All helices in $M_{k}$ are congruent to a helix of this form, except
geodesics, circles and, in the hyperbolic case, horocycles. Those with $%
\tau=0$ are hypercycles in totally geodesic hyperbolic planes in $H^3$; for
them, $\kappa <1$ holds.

\begin{definition}
\label{radiusEtc}Let $M$ be a complete oriented three-dimensional Riemannian
manifold of constant sectional curvature $k=0,1,-1$ and write $M=\Gamma
\backslash M_{k}$, as usual. A helix $h$ in $M$ with radius $r>0$, angular
speed $\mu $ \emph{(}with $\mu \neq 0$ if $k=0$\emph{)} and axis
with speed $c>0$ is by definition the projection to $M$ of a curve in $M_{k}$
congruent to $H$ by an element $g$ of $G_{k}$. The axis of $h$ is the
projection to $M$ of the geodesic $g\circ A$, where $A$ is the axis of $H$.
\end{definition}

The axis of $h$ is well defined, since although $h$ has as many lifts to $%
M_{k}$ as $\left\vert \Gamma \right\vert $, the element $g$ of $G_{k}$
referred to above is unique, since $G_{k}$ acts simply transitively on $%
SO\left( M_{k}\right) $ and a helix is determined by $F\left( 0\right) $
(the initial value of its Frenet frame), its curvature and torsion.

\subsection{Curvature and torsion of helices\nopunct}

\mbox{ }

\smallskip

We give a mostly self-contained presentation of the topics of this
subsection, although there is some overlapping with \cite{CS}. Unlike in
that article, here we deal in part with the three cases $k=0,1,-1$, use the
more suitable parameter $\mu $ (angular speed) and need stronger properties.
Also, some arguments are simpler.

\begin{lemma}
\label{HelixBasico}Let $k=0,1,-1$ and let $V\in \mathfrak{g}_{k}$. Then the
curve $h\left( t\right) =e^{tV}e_{0}$ is a helix. Suppose that $h$ has unit
speed and let $\kappa $ be the curvature of $h$. If $\kappa >0$, let $N$ be
the normal vector field and $\tau $ the torsion of $h$. Then 
\begin{equation}
\kappa =\left\Vert V^{2}e_{0}+ke_{0}\right\Vert \text{\qquad and \qquad }%
\kappa ^{2}\tau =\left\langle V^{3}e_{0},Ve_{0}\times \left(
V^{2}e_{0}+ke_{0}\right) \right\rangle \text{.}  \label{kappaAndTau}
\end{equation}
\end{lemma}

\begin{proof}
The curve $h$ is a helix since it is the orbit of a monoparametric group of
isometries of $M_{k}$. We have that $h^{\prime }\left( t\right)
=e^{tV}h^{\prime }\left( 0\right) =e^{tV}Ve_{0}$. Assuming that $h$ has
constant unit speed, $\left\langle Ve_{0},Ve_{0}\right\rangle =1 $ and $%
T\left( t\right) =h^{\prime }\left( t\right) $ hold. Similarly, $h^{\prime
\prime }\left( t\right) =e^{tV}h^{\prime \prime }\left( 0\right)
=e^{tV}V^{2}e_{0}$.

Recall that for $k=1,-1$, the covariant derivative of a vector field $v$
along $h$ at $t$ is the orthogonal projection of $v^{\prime }\left( t\right) 
$ onto $T_{h\left( t\right) }M_{k}=h\left( t\right) ^{\perp }$. The formula 
\begin{equation}
\tfrac{Dv}{dt}\left( t\right) =v^{\prime }\left( t\right) -k\left\langle
v^{\prime }\left( t\right) ,h\left( t\right) \right\rangle _{k}h\left(
t\right)  \label{CovDer}
\end{equation}%
is also valid when the inner product (\ref{InnerProduct}) is degenerate.
Note that $kV$ is skew-symmetric since $V\in \mathfrak{g}_{\kappa }$. We
compute 
\begin{equation*}
k\left\langle h^{\prime \prime }\left( 0\right) ,h\left( 0\right)
\right\rangle _{\kappa }=k\left\langle V^{2}e_{0},e_{0}\right\rangle
_{\kappa }=-k\left\langle Ve_{0},Ve_{0}\right\rangle _{\kappa }=-k
\end{equation*}%
and then, by (\ref{CovDer}),%
\begin{equation}
\kappa N\left( 0\right) =\tfrac{Dh^{\prime }}{dt}\left( 0\right) =h^{\prime
\prime }\left( 0\right) -k\left\langle h^{\prime \prime }\left( 0\right)
,h\left( 0\right) \right\rangle h\left( 0\right) =V^{2}e_{0}+ke_{0}\text{.}
\label{kappaN(0)}
\end{equation}%
Hence, the assertion regarding $\kappa $ is true if $\kappa >0$.

Since the covariant derivative commutes with the isometry $e^{tV}$, we have
that 
\begin{equation*}
\kappa N\left( t\right) =\tfrac{Dh^{\prime }}{dt}\left( t\right) =e^{tV}%
\tfrac{Dh^{\prime }}{dt}\left( 0\right) =e^{tV}\left(
V^{2}e_{0}+ke_{0}\right) \text{.}
\end{equation*}%
Hence, $\left. \frac{d}{dt}\right\vert _{0}\kappa N\left( t\right)
=V^{3}e_{0}+kVe_{0}$, which belongs to $T_{e_{0}}M_{k}$, and so it coincides
with $\kappa \tfrac{DN}{dt}\left( 0\right) $. Now,%
\begin{equation*}
\kappa B\left( 0\right) =T\left( 0\right) \times \kappa N\left( 0\right)
=Ve_{0}\times \left( V^{2}e_{0}+ke_{0}\right) \text{,}
\end{equation*}%
and so, again by the Frenet equations,%
\begin{equation*}
\kappa ^{2}\tau =\left\langle \kappa \tfrac{DN}{dt}\left( 0\right) ,\kappa
B\left( 0\right) \right\rangle =\left\langle V^{3}e_{0}+kVe_{0},Ve_{0}\times
\left( V^{2}e_{0}+ke_{0}\right) \right\rangle \text{,}
\end{equation*}%
which yields the second expression in (\ref{kappaAndTau}) since $%
\left\langle u,u\times v\right\rangle =0$ for all $u,v$.
\end{proof}

\medskip

\begin{proof}[Proof of Proposition \protect\ref{helixKappaTau}]
Since the canonical projection $M_{k}\rightarrow M=\Gamma \backslash M_{k}$
is a direct local isometry and $G_{k}$ acts on $M_{k}$ by direct isometries,
it suffices to show that the helix $h\left( t\right) =e^{tV\left( c,\mu
\right) }p$ with $V\left( c,\mu \right) $ as in (\ref{VcMu}) and $p=\left(
\cos _{k}r,0,\sin _{k}r,0\right) $ satisfies the stated conditions.

In order to apply Lemma \ref{HelixBasico}, we translate $h$ to a helix $\bar{%
h}$ with $\bar{h}\left( 0\right) =e_{0}$. Let $g$ be the translation of $%
M_{k}$ along the geodesic joining $e_{0}$ with $p$, mapping the former point
to the latter, that is, $g$ is the unique element of $G_{k}$ with $g\left(
e_{0}\right) =p$ and fixing $e_{1}$ and $e_{3}$; in particular, $g\left(
e_{2}\right) =-k\sin _{k}r~e_{0}+\cos _{k}r~e_{2}$. Then, 
\begin{equation*}
\bar{h}\left( t\right) =g^{-1}h\left( t\right) =g^{-1}e^{V\left( c,\mu
\right) }p=g^{-1}e^{tV\left( c,\mu \right) }g\left( g^{-1}p\right)
=e^{tV}e_{0}\text{,}
\end{equation*}%
where $V=g^{-1}V\left( c,\mu \right) g$. A straightforward computation yields%
\begin{equation}
V=c\left( 
\begin{array}{cccc}
0 & -k\cos _{k}r & 0 & -k\mu \sin _{k}r \\ 
\cos _{k}r & 0 & -k\sin _{k}r & 0 \\ 
0 & k\sin _{k}r & 0 & -\mu \cos _{k}r \\ 
\mu \sin _{k}r & 0 & \mu \cos _{k}r & 0%
\end{array}%
\right) \text{.}  \label{matrixW}
\end{equation}

Since $\bar{h}^{\prime }\left( 0\right) =Ve_{0}=c\cos _{k}r~e_{1}+c\mu \sin
_{k}r~e_{3}$, the first assertion follows.

Using (\ref{matrixW}) and the unit speed condition (\ref{UnitSpeed}) on $h$,
we compute%
\begin{equation*}
V^{2}e_{0}=V\left( Ve_{0}\right) =-ke_{0}+c^{2}\left( k-\mu ^{2}\right) \cos
_{k}r\sin _{k}r~e_{2}\text{.}
\end{equation*}%
Hence, by the first expression in (\ref{kappaAndTau}), 
\begin{equation}
\kappa ^{2}=c^{4}\left( k-\mu ^{2}\right) ^{2}\cos _{k}^{2}r\sin _{k}^{2}r%
\text{.}  \label{curvatureProv}
\end{equation}

A straightforward computation yields%
\begin{equation*}
Ve_{0}\times \left( V^{2}e_{0}+ke_{0}\right) =c^{3}\left( \mu ^{2}-k\right)
\cos _{k}r\sin _{k}r\left( \mu \sin _{k}r~e_{1}-\cos _{k}r~e_{3}\right) 
\text{.}
\end{equation*}%
Again by (\ref{matrixW}), we have that 
\begin{eqnarray*}
V^{3}e_{0} &=&V\left( V^{2}e_{0}\right) =-kVe_{0}+c^{2}\left( k-\mu
^{2}\right) \cos _{k}r\sin _{k}r~Ve_{2} \\
&=&-kVe_{0}+c^{3}\left( k-\mu ^{2}\right) \cos _{k}r\sin _{k}r\left( -k\sin
_{k}r~e_{1}+\mu \cos _{k}r~e_{3}\right) \text{.}
\end{eqnarray*}%
Then, by the second expression in (\ref{kappaAndTau}) (since $\left\langle
Ve_{0},Ve_{0}\times u\right\rangle =0$ for all $u$),%
\begin{equation*}
\kappa ^{2}\tau =c^{6}\left( k-\mu ^{2}\right) ^{2}\cos _{k}^{2}r\sin
_{k}^{2}r\left\langle -k\sin _{k}r~e_{1}+\mu \cos _{k}r~e_{3},\cos
_{k}r~e_{3}-\mu \sin _{k}r~e_{1}\right\rangle \text{.}
\end{equation*}

Now, by (\ref{curvatureProv}) we get%
\begin{equation*}
\kappa ^{2}\tau =\kappa ^{2}c^{2}\mu \left( \cos _{k}^{2}r+k\sin
_{k}^{2}r\right) =\kappa ^{2}c^{2}\mu \text{.}
\end{equation*}%
Thus, $\tau =c^{2}\mu $ if $\kappa >0$.

Finally, we show that $\kappa ^{2}$ is equal to the first expression in (\ref%
{CurvAndTorsion}). We write $\kappa ^{2}$ in (\ref{curvatureProv}) as $%
\kappa ^{2}=AB$ with $A=c^{2}\left( \mu ^{2}-k\right) \sin _{k}^{2}r$ and $%
B=c^{2}\left( \mu ^{2}-k\right) \cos _{k}^{2}r$. By the identity $\cos
_{k}^{2}+k\sin _{k}^{2}=1$ and the unit speed condition (\ref{UnitSpeed}),
we have%
\begin{eqnarray*}
A &=&c^{2}\mu ^{2}\sin _{k}^{2}r-c^{2}k\sin _{k}^{2}r=1-c^{2}\cos
_{k}^{2}r-c^{2}k\sin _{k}^{2}r \\
&=&1-c^{2}\left( \cos _{k}^{2}r+k\sin _{k}^{2}r\right) =1-c^{2}\text{.}
\end{eqnarray*}%
Similarly,%
\begin{eqnarray*}
B &=&c^{2}\mu ^{2}\cos _{k}^{2}r-kc^{2}\cos _{k}^{2}r=c^{2}\mu ^{2}\left(
1-k\sin _{k}^{2}r\right) -kc^{2}\cos _{k}^{2}r \\
&=&c^{2}\mu ^{2}-kc^{2}\left( \mu ^{2}\sin _{k}^{2}r+\cos _{k}^{2}r\right)
=c^{2}\mu ^{2}-k\text{.}
\end{eqnarray*}%
Consequently, $\kappa ^{2}$ is as stated.
\end{proof}

We comment that using similar arguments one can obtain that $\kappa =\left(
1-c^{2}\right) \cot _{k}r$, a shorter expression for the curvature, but less
suitable for our purposes.

\begin{lemma}
\label{HelixF(0)=I}Let $k=0,1,-1$. Let $Z=\left( 
\begin{array}{cc}
0 & -kx^{T} \\ 
x & L_{z}%
\end{array}%
\right) \in \mathfrak{g}_{\kappa }$, with $x,z\in \mathbb{R}^{3}$ and let $%
\kappa ,\tau \in \mathbb{R}$, with $\tau =0$ if $\kappa >0$. 
Then the following assertions are equivalent:

\smallskip

\emph{a)} The helix $h\left( t\right) =e^{tZ}e_{0}$ has unit speed, curvature $\kappa $ and
torsion $\tau$, and satisfies that $F\left( 0\right) $ is the identity on $%
T_{e_{0}}M_{k}$.

\smallskip

\emph{b)} For $Z$, $x=e_{1}$ and $z=\tau e_{1}+\kappa e_{3}$ hold.
\end{lemma}

\begin{proof}
Suppose that $h$ has the stated properties. We have by the hypothesis that 
\begin{equation*}
e_{1}=F\left( 0\right) e_{1}=h^{\prime }\left( 0\right) =Ze_{0}=x\text{.}
\end{equation*}%
Then, by (\ref{kappaN(0)}) and the definition of $Z$ we have%
\begin{eqnarray*}
\kappa e_{2} &=&\kappa F\left( 0\right) e_{2}=\kappa N\left( 0\right)
=Z^{2}e_{0}+ke_{0}=Z\left( Ze_{0}\right) +ke_{0} \\
&=&Ze_{1}+ke_{0}=-ke_{0}+z\times e_{1}+ke_{0}=z\times e_{1}\text{.}
\end{eqnarray*}%
In particular, $Z^{2}e_{0}=-ke_{0}+\kappa e_{2}$. Setting $%
z=ae_{1}+be_{2}+ce_{3}$, we have then 
\begin{equation*}
\kappa e_{2}=z\times e_{1}=ce_{2}-be_{3}\text{,}
\end{equation*}%
which implies that $z=ae_{1}+\kappa e_{3}$. Now,%
\begin{eqnarray*}
Z^{3}e_{0} &=&Z\left( -ke_{0}+\kappa e_{2}\right) =-kZe_{0}+\kappa
Ze_{2}=-ke_{1}+\kappa \left( z\times e_{2}\right) \\
&=&-ke_{1}+\kappa \left( ae_{1}+\kappa e_{3}\right) \times e_{2}=-\left(
k+\kappa ^{2}\right) e_{1}+a\kappa e_{3}\text{.}
\end{eqnarray*}

Then, by the second identity in (\ref{kappaAndTau}),%
\begin{eqnarray*}
\kappa ^{2}\tau &=&\left\langle Z^{3}e_{0},Ze_{0}\times \left(
Z^{2}e_{0}+ke_{0}\right) \right\rangle =\left\langle -\left( k+\kappa
^{2}\right) e_{1}+a\kappa e_{3},e_{1}\times \kappa e_{2}\right\rangle \\
&=&\left\langle a\kappa e_{3},\kappa e_{3}\right\rangle _{\kappa }=\kappa
^{2}a\text{.}
\end{eqnarray*}%
Thus, $a=\tau $ if $\kappa \neq 0$, as desired. The converse follows the
same lines, in a simpler manner.
\end{proof}

\begin{proof}[Proof of Proposition \protect\ref{PeriodicHelix}]
The proposition is very similar to Lemma 8 in \cite{CS}, so we only comment
on the differences. The lemma does not deal with the (simpler) flat case;
also, the notation and the parameters vary: $E=\alpha $, $T_{0}=L$, $T=\ell
/c$ is the period of $\alpha $, $\phi _{t}=\exp \left( tV\left( c,\mu
\right) \right) $ (see (\ref{VcMu})) and $b=c\mu $ (in the present paper we
have recognized the convenience of introducing the angular speed of a helix: 
$\mu $ is geometrically more relevant than $b$).
\end{proof}

\section{Proofs of the results\label{Section6}}

\subsection{Sub-Riemannian geodesics\nopunct}

\mbox{ }

\smallskip

For $k=0,1,-1$ and $\lambda \neq k^{2}$, let $G_{k}$ be endowed with the $%
\lambda $-screw sub-Riemannian structure defined in the introduction. We
recall Proposition 1.4 in \cite{HMS}, which describes the maximal
sub-Riemannian geodesics through the identity.

\begin{proposition}
\label{PropHMS} \emph{\cite{HMS}} A curve in $G_{\kappa }$ is a
sub-Riemannian geodesic through the identity if and only if it equals $%
\gamma _{x,y}$ for some $x,y\in \mathbb{R}^{3}$, where, for all $t$, 
\begin{equation}
\gamma _{x,y}\left( t\right) =\exp \left( t\left( 
\begin{array}{cc}
0 & -\kappa x^{T} \\ 
x & L_{\lambda x+y}%
\end{array}%
\right) \right) \exp \left( t\left( 
\begin{array}{cc}
0 & 0 \\ 
0 & -L_{y}%
\end{array}%
\right) \right) \text{.}  \label{gammaXY}
\end{equation}
\end{proposition}

\medskip

\begin{proof}[Proof of Theorem \protect\ref{geodesica}]
We may consider only the case $M=M_{k}$, since the canonical projection $%
SO\left( M_{k}\right) \rightarrow SO\left( M\right) =\Gamma \backslash
SO\left( M_{k}\right) $ is a local sub-Riemannian isometry. First we prove
that $\gamma $ as in (\ref{s-Rgeodesic}) is a geodesic of $\left( SO\left(
M_{k}\right) ,\mathcal{D}^{\lambda }\right) $. By left $G_{k}$-invariance
and right invariance by $SO\left( 3\right) $, we may suppose without loss of
generality that $h\left( 0\right) =e_{0}$, $F\left( 0\right) $ is the
identity on $T_{e_{0}}M_{k}\equiv \mathbb{R}^{3}$ and $O=I$.

By Lemma \ref{HelixF(0)=I}, $h\left( t\right) =e^{tZ}e_{0}$, where%
\begin{equation*}
Z=\left( 
\begin{array}{cc}
0 & -ke_{1}^{T} \\ 
e_{1} & L_{\tau e_{1}+\kappa e_{3}}%
\end{array}%
\right) \in \mathfrak{g}_{\kappa }\text{.}
\end{equation*}%
We write $\tau e_{1}+\kappa e_{3}=\lambda e_{1}+y$ with $y=\left( \tau
-\lambda \right) e_{1}+\kappa e_{3}$ and set $Y=$ diag$~\left(
0,L_{y}\right) \in \mathfrak{so}\left( 3\right) $.

By Proposition \ref{PropHMS}, $\delta \left( t\right) =e^{tZ}e^{-tY}$ is a
geodesic in $G_{k}$ such that $\delta \left( t\right)
e_{0}=e^{tZ}e_{0}=h\left( t\right) $, since $e^{-tY}\in SO\left( 3\right) $
fixes $e_{0}$ for all $t$. Now, as $F\left( 0\right) $ is the identity on $%
T_{e_{0}}M_{k}\equiv \mathbb{R}^{3}$, the corresponding sub-Riemannian
geodesic of $SO\left( M_{k}\right) $ is equal to $\Phi \left( \delta \left(
t\right) \right) =\left( h\left( t\right) ,b\left( t\right) \right) $, with 
\begin{equation}
b\left( t\right) =e^{tZ}F\left( 0\right) e^{-tY}=F\left( t\right) \exp
_{SO\left( 3\right) }\left( -tL_{y}\right) =F\left( t\right) \operatorname{Rot}%
\left( \left( \lambda -\tau \right) e_{1}-\kappa e_{3},t\right) \text{,}
\label{b(t)}
\end{equation}%
as desired. Also, it is well-known that in our situation, if $\sigma $ is a
sub-Riemannian geodesic, so is the curve $t\mapsto \sigma \left( dt\right) $
for $d\geq 0$.

For the converse, suppose that $\gamma $ is a geodesic of $\left( SO\left(
M_{k}\right) ,\mathcal{D}^{\lambda }\right) $ projecting to a helix $h$ with
Frenet frame $F$. By the $\left( G_{k}\times SO\left( 3\right) \right) $%
-invariance we may assume that $h\left( 0\right) =e_{0}$ and $\gamma \left(
0\right) $ is equal to $F\left( 0\right) $ and also to the identity on $%
T_{e_{0}}M_{k}\equiv \mathbb{R}^{3}$. Then $\left( \Phi ^{-1}\gamma \right)
\left( 0\right) =I$ and so, by Proposition \ref{PropHMS}, $\gamma =\Phi
\circ \gamma _{x,y},$ with $\gamma _{x.y}$ as in (\ref{gammaXY}) for some $%
x,y\in \mathbb{R}^{3}$. Hence, $h\left( t\right) =\exp \left( tZ\right)
e_{0} $, where $Z$ is as in Lemma \ref{HelixF(0)=I}, with $z=\lambda x+y$.
If $h$ has curvature $\kappa $ and torsion $\tau $, by the same lemma, $%
z=\tau e_{1}+\kappa e_{3}$. The argument continues as in (\ref{b(t)}).
\end{proof}

\subsection{Sub-Riemannian length spectrum\label{ProofsSRLS} \nopunct}

\mbox{ }

\smallskip

\begin{proof}[Proof of Proposition \protect\ref{fundamentalProp}]
The assertion in (a) is clearly true. Next we prove (b). Since $\alpha $ has
unit speed, its period is $\ell $. Let $F$ be a Frenet frame for $\alpha $
as in (\ref{FrenetGeo}). By definition of the holonomy of $\alpha $, for all 
$t$ we have 
\begin{equation}
F\left( t+\ell \right) =\operatorname{Rot}\left( \alpha ^{\prime }\left( t\right)
,\theta \right) F\left( t\right) =F\left( t\right) \operatorname{Rot}\left(
e_{1},\theta \right) \text{,}  \label{FR}
\end{equation}%
since $F\left( t+\ell \right) e_{2}=N\left( t+\ell \right) =\mathcal{T}%
_{t,t+\ell }N\left( t\right) =\mathcal{T}_{t,t+\ell }F\left( t\right) e_{2}$%
, and similarly for $e_{3}$ and $B$.

Suppose that $\gamma $ is periodic with period $T$. Then, by (a), $T=n\ell $
for some positive integer $n$. For brevity we call $R_{\kappa ,\tau
}^{\lambda }\left( t\right) =\operatorname{Rot}\left( \left( \lambda -\tau
,0,-\kappa \right) ,t\right) $. By (\ref{FR}), we have%
\begin{eqnarray*}
\left( \alpha \left( t\right) ,F\left( t\right) R_{0,0}^{\lambda }\left(
t\right) O\right) &=&\gamma \left( t\right) =\gamma \left( t+n\ell \right) \\
&=&\left( \alpha \left( t+n\ell \right) ,F\left( t+n\ell \right)
R_{0,0}^{\lambda }\left( t+n\ell \right) O\right) \\
&=&\left( \alpha \left( t\right) ,F\left( t\right) \operatorname{Rot}\left(
e_{1},n\theta \right) R_{0,0}^{\lambda }\left( t\right) R_{0,0}^{\lambda
}\left( n\ell \right) O\right) \text{.}
\end{eqnarray*}
for all $t$. 
Hence, $I=\operatorname{Rot}\left( e_{1},n\theta \right) R_{0,0}^{\lambda }\left(
n\ell \right) =\operatorname{Rot}\left( e_{1},n\theta +\lambda n\ell \right) $ and
so there exists an integer $m$ such that 
\begin{equation}
n\left( \theta +\lambda \ell \right) =2m\pi \text{.}  \label{CondicionCorta}
\end{equation}%
Next we see that $n$ and $m$ are coprime. Again by (\ref{FR}) we have that%
\begin{eqnarray*}
\gamma \left( t+\tfrac{n}{\left( n,m\right) }\ell \right) &=&\left( \alpha
\left( t+\tfrac{n}{\left( n,m\right) }\ell \right) ,F\left( t+\tfrac{n}{%
\left( n,m\right) }\ell \right) R_{0,0}^{\lambda }\left( t+\tfrac{n}{\left(
n,m\right) }\ell \right) O\right) \\
&=&\left( \alpha \left( t\right) ,F\left( t\right) \operatorname{Rot}\left( e_{1},%
\tfrac{n}{\left( n,m\right) }\theta \right) R_{0,0}^{\lambda }\left( t+%
\tfrac{n}{\left( n,m\right) }\ell \right) O\right) \text{,}
\end{eqnarray*}%
which, by (\ref{CondicionCorta}) is equal to%
\begin{equation*}
\left( \alpha \left( t\right) ,F\left( t\right) \operatorname{Rot}\left(
e_{1},\lambda t+\tfrac{2m}{\left( n,m\right) }\pi \right) O\right) =\left(
\alpha \left( t\right) ,F\left( t\right) \operatorname{Rot}\left( e_{1},\lambda
t\right) O\right) =\gamma \left( t\right) \text{.}
\end{equation*}%
Thus, $\left( m,n\right) =1$, since otherwise, $n\ell $ would not be the
period of $\gamma $. The last assertion follows from the definition of the
sub-Riemannian structure in (\ref{s-Rmetric}).

The converse follows the same lines of reasoning.

\smallskip

c) Since $h$ has unit speed, its period is $L$, as well as the period of its
Frenet frame (as $\kappa >0$, by hypothesis). Suppose that $\gamma $ is
periodic with period $T$. Then $T=nL$ for some positive integer $n.$ For all 
$t$ we have%
\begin{eqnarray*}
\left( \alpha \left( t\right) ,F\left( t\right) R_{\kappa ,\tau }^{\lambda
}\left( t\right) O\right) &=&\gamma \left( t\right) =\gamma \left( t+T\right)
\\
&=&\left( \alpha \left( t+T\right) ,F\left( t+T\right) R_{\kappa ,\tau
}^{\lambda }\left( t+T\right) O\right) \\
&=&\left( \alpha \left( t\right) ,F\left( t\right) R_{\kappa ,\tau
}^{\lambda }\left( t\right) R_{\kappa ,\tau }^{\lambda }\left( T\right)
O\right) \text{.}
\end{eqnarray*}%
Hence,%
\begin{equation*}
I=R_{\kappa ,\tau }^{\lambda }\left( T\right) =\exp _{SO\left( 3\right)
}\left( TL_{\left( \lambda -\tau ,0,-\kappa \right) }\right) \text{,}
\end{equation*}%
which implies that $T\sqrt{\kappa ^{2}+\left( \lambda -\tau \right) ^{2}}%
=2m\pi $ for some positive integer $m$.

The preceding arguments imply the validity of (\ref{lengthNL}).
\end{proof}

\begin{proof}[Proof of Proposition \protect\ref{PropoBasic}]
The assertion in (a) is a consequence of Proposition \ref{fundamentalProp}
(a), since circles are exactly the periodic helices in $M_{k}$. Now we prove
(b). It is well known that the length $L$ of a circle of radius $r$ in $%
M_{k} $ is equal to $2\pi \sin _{k}r$ and its curvature is $\kappa =\cot
_{k}r$. Hence, Equation (\ref{FundamentalEqShort}) reads 
\begin{equation*}
n2\pi \sin _{k}r\sqrt{\lambda ^{2}+\cot _{k}^{2}r}=2m\pi \text{,}
\end{equation*}%
which is not difficult to translate into (\ref{sinRadius}), using that $\cos
_{k}^{2}+k\sin _{k}^{2}=1$. The last assertion follows from (\ref{lengthNL}).
\end{proof}

\begin{proof}[Proof of Theorem \protect\ref{TeoFormulaLarga}]
Suppose that $h$ is the projection of a periodic sub-Riemannian geodesic $%
\gamma $. By the arguments before the statement of the theorem, there exist
coprime positive integers $n,m$ such that $c^{2}$ satisfies Equation (\ref%
{EqN}). Now, $0<c^{2}<1$ by Proposition \ref{helixKappaTau}, and so the
coefficient of $c^{2}$ is positive and bigger than $\lambda ^{2}-k$. Calling 
$U$ the first term of the coefficient of $c^{2}$, the condition translates
into%
\begin{equation*}
U+\left( \lambda ^{2}-k\right) -\left( \mu -\lambda \right) ^{2}>\lambda
^{2}-k\text{,}
\end{equation*}%
which is equivalent to $U>\left( \mu -\lambda \right) ^{2}$ and also to (\ref%
{conditionMN}). One solves for $\sin _{k}^{2}r$ using the identity $\cos
_{k}^{2}+k\sin _{k}^{2}=1$ and obtains (\ref{radiusBig}).

By (\ref{lengthNL}) and the first expression in (\ref{DisplayPer}), the
length of $\gamma $ is given by $nL=nq\ell /c$ and turns out to be as in (%
\ref{lengthBig}) solving for $c$ in (\ref{EqN}). The converse follows the
same lines of reasoning.
\end{proof}

\begin{proof}[Proof of Theorem \protect\ref{unionSpectrum}]
We first observe that the length of a periodic curve does not depend on the
parametrization and hence, in the case of a sub-Riemannian geodesic of $%
\left( SO\left( M\right) ,\mathcal{D}^{\lambda }\right) $, we may assume
that its projection to $M$ (which is a helix) has unit speed.

Proposition \ref{fundamentalProp} (b) and Theorem \ref{TeoFormulaLarga} give
the lengths of all periodic sub-Riemannian geodesics projecting to periodic
geodesics and to periodic helices with axes in $M$. The remaining helices
are either circles and horocycles (that is, curves in $H^{3}$ with $\tau =0$
and $\kappa =1$).

On the one hand, the length spectrum of $SO\left( M_{k}\right) $ is
contained in the length spectrum of $SO\left( M\right) $, since the
projection of the former to the latter is a local isometry. Besides, a
circle in $M_{k}$ cannot project to a periodic curve with smaller period,
since this would imply the existence of an elliptic element of $\Gamma $ (a
rotation around the center) with finite order; this contradicts that $\Gamma 
$ is the fundamental group of a smooth manifold. Then, Proposition \ref%
{PropoBasic} (b) provides all the lengths of sub-Riemannian geodesics of $%
SO\left( M\right) $ projecting to circles.

On the other hand, a horocycle cannot project to a periodic curve, since
otherwise there would be a parabolic element of $\Gamma $ translating it,
and this is not possible, since the manifold has positive injectivity radius.
\end{proof}


\noindent Marcos Salvai\newline
\noindent FAMAF (Universidad Nacional de C\'ordoba) and CIEM (Conicet) \newline
\noindent Ciudad Universitaria, X5000HUA C\'{o}rdoba, Argentina \newline
\noindent salvai@famaf.unc.edu.ar

\end{document}